\def\hc{\mbox{HC}}
\def\chaa{C_{\omega_1\omega}(\aaq)}
\def\Col{\mathop{\rm Col}}

\def\Lim{\mathop{\rm Lim}}

\def\Tr{{\mathop{\rm Tr}}}
\def\tr{\mathop{{\rm Tr}}}
\def\rud{\mbox{rud}}

\def\zfc{\mbox{{ZFC}}}
\def\zf{\mbox{{ZF}}}
\def\CAA{C(\hspace{-1pt}\aaq\hspace{-1pt})}
\def\aaq{\mathop{{\mbox{\tt a\hspace{-0.4pt}a}}}}
\def\aaqs{\mathop{\mbox{\tt a\hspace{-0.4pt}a\hspace{0.6pt}$s$\hspace{0.6pt}}}}

\def\MM{\mbox{\rm MM}}

scaled \magstephalf 
\def\restriction{\upharpoonright}
\def\rest{\upharpoonright}

\def\bx{\vec{x}}

\def\bt{\vec{t}}

\def\ba{\vec{a}}

\def\Pw{{\cal P}}
\def\On{{\rm On}}

\newcommand{\F}{{\cal F}}

\mathchardef\bfSigma="0606 \mathchardef\bfPi="0605
\mathchardef\bfDelta="0601

\def\rest{\mathord{\restriction}} % This is from AMS symbols,
\def\rest{\mathord{\restriction}}
\newcommand{\open}{\Bbb}

\newcommand{\oF}{{\open F}}

\newcommand{\oP}{{\open P}}

\newcommand{\oR}{{\open R}}

\newcommand{\oU}{{\open U}}
\newcommand{\oT}{{\open T}}
\makeatletter

\def\bracketdown#1{\mathop{\vbox{\ialign{##\crcr\noalign{\kern2\p@}
\downbracketfill\crcr\noalign{\kern2\p@\nointerlineskip}
$\hfil\displaystyle{#1}\hfil$\crcr}}}\limits}

\def\bracketup#1{\mathop{\vbox{\ialign{##\crcr\noalign{\kern1\p@}
\upbracketfill\crcr\noalign{\kern1\p@\nointerlineskip}
$\hfil\displaystyle{#1}\hfil$\crcr}}}\limits}

\def\upbracketfill{$\m@th
\makesm@sh{\llap{\vrule\@height2\p@\@width.4\p@}}%
\leaders\vrule\@height.4\p@\hfill
\makesm@sh{\rlap{\vrule\@height2\p@\@width.4\p@}}$}

\def\downbracketfill{$\m@th
\makesm@sh{\llap{\vrule\@height.4\p@\@depth1.6\p@\@width.4\p@}}%
\leaders\vrule\@height.4\p@\hfill
\makesm@sh{\rlap{\vrule\@height.4\p@\@depth1.6\p@\@width.4\p@}}$}
\makeatother

\def\CH{\text{CH}}
\def\GCH{\text{GCH}}

\def\\(\two\){Eloise}

\def\two{\mbox{$\mathbf{II}$}}

\def\phi{\varphi}
\def\hod{\mbox{\rm HOD}}

\documentclass[12pt]{article} %%%%%%
\usepackage{amsthm} 
\usepackage{graphicx} 
\usepackage{amssymb,latexsym,times,amsmath,bm}
\newcommand{\psfrag}[2]{}
\usepackage{makeidx,proof,color}
\usepackage{verbatim,lineno,comment}
\usepackage[utf8]{inputenc}
\DeclareMathAlphabet{\mathpzc}{OT1}{pzc}{m}{it}
 \theoremstyle{plain}% default
  \newtheorem{theorem}{Theorem}[section]
  \newtheorem{lemma}[theorem]{Lemma}
  \newtheorem{corollary}[theorem]{Corollary}
  
  \newtheorem{conjecture}[theorem]{Conjecture}

  \theoremstyle{definition}
  \newtheorem{definition}[theorem]{Definition}

  \theoremstyle{remark}

 \newtheorem*{claim}{Claim}
 \newtheorem*{subclaim}{Subclaim}

\def\LL{\mathcal{L}}
\def\cto{C^2(\omega)}
\def\ctoaa{C^2(\omega,\aaq)}
\def\ctop{C^{2,+}(\omega)}

\begin{document}
%

%\LARGE

\author{Menachem Magidor\\ Jerusalem \and Jouko V\"a\"an\"anen\\ Helsinki and Amsterdam}

\title{New inner models from second order logics.\thanks{We are grateful to the American Institute of Mathematics for its support. This project has also received  funding from the European Research Council (ERC) under the
European Union’s Horizon 2020 research and innovation programme (grant agreement No
101020762).}
}
  
\maketitle

\begin{abstract}
    $\mathcal{L}$.We define a new inner  model $\cto$ based on the fragment of second order logic in which second order variables range over countable subsets of the domain. We compare $\cto$ to the previously studied inner model $\CAA$. We argue that $\cto$ appears to be a much bigger inner model than $\CAA$, although this cannot be literally true in ZFC alone. However, we conjecture that it follows from large cardinal assumptions. For example, assuming large cardinals, $\cto$ contains, for every $n$, an inner model with $n$ Woodin cardinals, while $\CAA$ contains, under the same assumption, no inner model with a Woodin cardinal. As to large cardinals in $\cto$, we show that, assuming a  Woodin limit of Woodin cardinals, the cardinal $\omega_1^V$ is Mahlo in $\cto$. A stronger result is proved for the combination $\ctoaa$ of $\CAA$ and $\cto$. We also show that the question whether $\hod_1$, a variant of $\hod$, defined in \cite{MR4290501}, is the same as $\hod$ cannot be decided on the basis of ZFC even if we add the assumption that there are supercompact cardinals.  
\end{abstract}

\section{Introduction}

Suppose  $\mathcal{L}$ is some extension of first order logic, as in  \cite{BF}.  For example, $\mathcal{L}$ could be an extension of first order logic by a generalized quantifier or by higher order quantification. The inner model $C(\mathcal{L})$, introduced in \cite{MR4290501}, obtains if in the definition of G\"odel's constructible hierarchy $L$ definability in first order logic is replaced by definability in the logic $\mathcal{L}$. For the Axiom of Choice to be true in the inner model $C(\mathcal{L})$ we have to make a further modification, essentially adding reference to the truth definition of $\mathcal{L}$, see Section~\ref{cto}. The most investigated new inner models  of  the type $C(\mathcal{L})$ so far have been $C(\mathcal{L}(Q^{cof}_\omega))$ (\cite{MR4290501}), denoted $C^*$, arising from the so-called cofinality quantifier $Q^{cof}_\omega$, and 
$C(\mathcal{L}(\aaq))$ (\cite{KMV2}), denoted $\CAA$, arising from the so-called stationary logic. We refer to \cite{MR4290501,KMV2,BF} for definitions of these logics.

If we build the inner model $C(\mathcal{L})$ from second-order logic
$\mathcal{L}^2$, obtaining $C(\mathcal{L}^2)$, we get exactly $\hod$, as proved in \cite{MR0281603}. If we build the inner model $C(\mathcal{L})$ from the infinitary logic $\mathcal{L}_{\omega_1\omega}$, we obtain\footnote{Formulas of $\mathcal{L}_{\omega_1\omega}$ are allowed to have a countable sequence of free variables.} the \emph{Chang model} $C(\mathcal{L}_{\omega_1\omega})$, denoted $C_{\omega_1\omega}$, i.e. the smallest transitive model of $\zf$ containing all the ordinals and closed under countable sequences \cite{MR280357}.

In this paper we consider two \emph{fragments} of second order logic and the inner models they give rise to. 
The first fragment   is $\mathcal{L}^2_\omega$, second order logic in which the bound second order variables range over countable (rather than all) subsets and relations on the domain. It is important to note that these subsets and relations are chosen from $V$ and assumed to be countable in $V$, and they need not be elements of the domain, nor countable in the sense of the domain. The logic $\mathcal{L}^2_\omega$ is notoriously non-axiomatizable and incompact, and in that respect quite far from $\mathcal{L}(Q^{cof}_\omega)$ and from stationary logic $\mathcal{L}(\aaq)$. Still, in our set theoretical approach, $\mathcal{L}^2_\omega$ behaves similarly as $\mathcal{L}(Q^{cof}_\omega)$ and $\mathcal{L}(\aaq)$.

The fragment $\mathcal{L}^2_\omega$ gives rise to an inner model 
\begin{equation}\label{ctodef}
    \cto=C(\mathcal{L}^2_\omega).
\end{equation}
Obviously $C^*\subseteq\cto$ but it is less clear what the relationship between $\cto$ and $\CAA$ is. In Section~\ref{cto} we draw similarities and differences between $\cto$ and $\CAA$. The upshot of the results is that $\cto$ appears to be, in the presence of sufficiently large cardinals, a ``stronger" inner model than $\CAA$. An  indication of the strength of $\cto$ is the fact that e.g. $M_1^\sharp$ is in $\cto$ but not in $\CAA$ (assuming a proper class of Woodin cardinals).

We consider also the combination
$C^2(\omega,\aaq)$ of $C(\aaq)$ and $\cto$, defined as $C(\mathcal{L}^2_\omega(aa))$, where $\mathcal{L}^2_\omega(aa)$ is the extension of $\mathcal{L}^2_\omega$ by the $aa$-quantifier of stationary logic. We show that some of the properties of $\CAA$ hold also for the potentially bigger inner model $C^2(\omega,\aaq)$, assuming a proper class of Woodin limits of Woodin cardinals.

The second fragment of second order logic we consider is  \emph{existential second order logic} $\Sigma^1_1$ consisting of formulas of the form
$\exists R_1\ldots\exists R_n\phi,$
where $\phi$ is first order and $R_1,\ldots,R_n$ are second order relation variables. The resulting inner model $$\hod_1=C(\Sigma^1_1),$$ which extends e.g. $C^*$, was already considered in \cite{MR4290501}. This model is the topic of Section~\ref{hod-one} where the main result says that the equation $\hod_1=\hod$ cannot be decided by large cardinal axioms.

\section{The inner model $C^2(\omega)$.\label{cto}}

We have defined in (\ref{ctodef}) above the inner model
    $\cto$.
    Clearly, $\cto$ satisfies $\zf$. For $\cto$ to satisfy $AC$ we   need to fold in the truth-definition in the definition  of $\cto$, i.e. we use the ``new" definition of the inner model $C(\mathcal{L})$ introduced in \cite{KMV2}, as it is not clear whether $\cto$ is adequate to truth in itself in the sense that would warrant the use of the older definition of \cite{MR4290501}. If we used the  definition of $C(\mathcal{L}^2_\omega)$ as in \cite{MR4290501}, we do not know whether the inner model satisfies the Axiom of Choice. 
Let $\Lim$ be the class of limit ordinals.  We define the hierarchy $(J'_\alpha)$, $\alpha\in\Lim$,
and the class $\Tr$,  by  transfinite double induction,  as follows:
$$\Tr=\{(\alpha,\phi(\ba)): (J'_\alpha,\in,\Tr\rest \alpha)\models\phi(\ba), \phi(\bx)\in\mathcal{L}^2_\omega,\ba\in J'_\alpha, \alpha\in\Lim\},$$
where $\Tr\rest \alpha=\{(\beta, \psi(\ba))\in \Tr:\beta\in \alpha\cap\Lim\},$
and
\def\arraystretch{1.3}
$$\label{J-hier}
\left\{\begin{array}{lcl}
J'_{0}&=&\emptyset\\
J'_{\alpha+\omega}&=&\rud_{\Tr}(J'_\alpha\cup\{J'_\alpha\})\\
J'_{\omega\nu}&=&\bigcup_{\alpha<\nu}J'_{\omega\alpha},\mbox{ for }\nu\in\Lim.
\end{array}\right.
$$
\def\arraystretch{1}
Here the rudimentary closure operation $\rud_{\Tr}$ includes the operation $x\mapsto x\cap \Tr$.
 The inner model $\cto$ is the class $\bigcup_{\alpha\in\Lim}J'_\alpha$. The inner model $\ctoaa$ is defined similarily.

 Instead of countable subsets, we can define $\cto$ by reference to countable sequences, and we obtain the same model.
 
 The  following inclusions  (in ZFC) are trivial: 
\begin{enumerate}
\item $C^*\subseteq \cto\subseteq C^2(\omega,\aaq)\subseteq\hod$.
\item $C^*\subseteq \CAA\subseteq C^2(\omega,\aaq)$.
\end{enumerate}

 The following questions about $\cto$ immediately suggest themselves, both consistently and under large cardinal assumptions:
 \begin{enumerate}
\item[Q1:] Does $\cto$ satisfy $\GCH$?
\item[Q2:] What is the relationship between $\cto$ and $\CAA$?
\item[Q3:] Is $\cto$ forcing absolute?
\item[Q4:] What large cardinals exist in $\cto$?

\end{enumerate} Similar questions can be posed also for $C^2(\omega,\aaq)$. 
We give partial answers to all four questions.

Let us first note the following rather simple fact, related to Q1, which still has interesting consequences:

\begin{theorem}
For any $\kappa\ge 2$, $|\mathcal{P}(\kappa)\cap \cto|\le (\kappa^\omega)^+$.
\end{theorem}

\begin{proof}
This is as the proof of Theorem 5.19 in \cite{MR4290501}. Let $(J'_\alpha)_{\alpha\in\On}$ be the hierarchy defining $\cto$ (see~\cite{KMV2}). Suppose $a\subseteq\kappa$ and $a\in J'_\xi$. Let $\mu$ be an uncountable regular cardinal bigger than $\xi$ and $\kappa$. Let $(M_\alpha)_{\alpha<\omega_1}$ be a chain of models such that 
\begin{enumerate}
\item $M_0\models a\in \cto$.
\item $\kappa\cup\{a\}\subseteq M_0$.
\item $|M_\alpha|\le\kappa^\omega$.
\item $M_\alpha\cup\{M_\alpha\}\subseteq M_{\alpha+1}$.
\item $M_\alpha\preccurlyeq_{\mathcal{L}^2(\omega)}H_\mu$.
\item $M_\nu=\bigcup_{\alpha<\nu}M_\alpha$ for limit $\nu$.
\end{enumerate}
For item 5 we use the Downward L\"owenheim-Skolem Theorem of $\mathcal{L}_{\omega_1\omega_1}$
(\cite{MR539973}) and the fact that  $\mathcal{L}^2_\omega$ is a sublogic of $\mathcal{L}_{\omega_1\omega_1}$.
Let $M=\bigcup_{\alpha<\omega_1}M_\alpha$ and $\pi:N\cong M$ with $N$ transitive. Note that $\pi^{-1}(a)=a\in N$ and $N^\omega\subseteq N$. Let $\zeta=N\cap\On$. It is easy to see by induction on $\alpha$ that $(J'_\alpha)^N=(J'_\alpha)^V$ for $\alpha<\zeta$. Hence $a\in J'_\alpha$ for some $\alpha<\zeta$. Since $|\zeta|\le \kappa^\omega$, $|J'_\zeta|\le
\kappa^\omega$.
\end{proof}

The interest of the above theorem from the point of view of question Q1 is rather limited as in large cardinal scenarios it is likely that $(\kappa^\omega)^+$ is quite large compared to  $2^\kappa$ of $\cto$. Conceivably, but we do not know how to prove this at the moment, $(\kappa^\omega)^+$ is even measurable in $\cto$ (but see Theorem~\ref{measurable} below). However, we can now say something about Q2: Recall that $|\mathcal{P}(\kappa)\cap C^* |\le (\kappa\cdot \omega_1)^+$ (\cite[Theorem 5.20]{MR4290501}) and that there is no provable upper bound (in $\zfc$) for $|\mathcal{P}(\kappa)\cap \CAA |$ (\cite[Theorem 3.6]{KMV2}).
We can have consistently, relative to the consistency of $\zf$,  $|\mathcal{P}(\aleph_1)\cap \CAA|=\aleph_3$, while at the same time $(\aleph_1)^{\CAA}=\aleph_1$ and  (by the above) $|\mathcal{P}(\aleph_1)\cap \cto|\le \aleph_2$. Hence in this forcing extension $\CAA\nsubseteq\cto$.
We can make the conclusion even stronger:

\begin{theorem}
 Con(ZF) implies Con($\cto\subsetneq\CAA$).
\end{theorem}

\begin{proof}
Assume $V=L$. Let $S_n$, $n<\omega$, be disjoint stationary subsets of $\omega_1$. We force  a Cohen subset $a$ of $\omega$. The inner model $\cto$ is still after the forcing just $L$, by the homogeneity of the forcing. The sets $S_n$ remain stationary under this forcing as the forcing is CCC. Next we code the Cohen real by shooting a club into the stationary set $\bigcup_{n\notin a}(\omega_1\setminus S_n)$. This forcing does not change $\cto$, because the forcing does not add new countable sets. So in the final model we still have $\cto=L$ but now 
$\forall n(n\in a\iff \mbox{$S_n$ is stationary})$, so $a\in \CAA$, and therefore $\CAA\ne L$.  
\end{proof}

 Theorem~\ref{lmu} below gives a similar result in the presence of a measurable cardinal.
%Question: How far can this be extended?
%Depends on the SQuaRE paper.
We do not know at the moment whether large cardinals imply $\CAA\subseteq\cto$, but there are indications in this direction (see, e.g., Theorems \ref{manywoodins}, \ref{realsofcaaincto},  and \ref{subsetsofomega1}   below).

We can make the following further remark concerning Q1:

\begin{theorem}
    Con(ZF) implies Con($\cto\models\neg\CH$). 
\end{theorem}

\begin{proof}
By    \cite{MR0465866} it is consistent, relative to the consistency of ZF that $\CH$ fails but there is a projective well-order $<^*$ of the reals. We show now that in this model the inner model $\cto$ contains all the reals.
%We can use transfinite induction on this well-order to prove that in the model of  \cite{MR0465866} the inner model $\cto$ contains all the reals. 
To this end, let us call a well-ordered sequence of reals \emph{nice}, if every element of the sequence is the first in the well-order $<^*$ after the preceding elements of the sequence. We can say in $\cto$ that a sequence in $\cto$ is nice in $V$. Of any two nice sequences one is always an initial segment of the other. Let $s$ be the  union of all nice sequences that are in $\cto$. Clearly, $s$ itself is nice and contains all the reals of $V$. 
\end{proof}

\begin{conjecture}
    For every $n$, $(\neg\CH)^{\cto}$ is consistent with $n$ Woodin cardinals, assuming the existence of $M_n$. 
    \end{conjecture}
    %Question: How much large cardinals can we have in this? As long as we have projective well-order of the reals and the inner model is included in $\cto$. For any $n$, can have $n$ Woodin cardinals.

 We proceed now to question Q3. We first note a simple fact:
\begin{lemma}\label{chang}
    $\cto\subseteq \hod^{C_{\omega_1\omega}}$.
\end{lemma} 

\begin{proof}
    If we  repeat the construction of $\cto$ inside ${C_{\omega_1\omega}}$, we obtain the same $\cto$, as can be easily seen by induction on the levels of the construction of $\cto$. Thus, $\cto=\cto^{{C_{\omega_1\omega}}}$. On the other hand,  $\cto\subseteq\hod$.
\end{proof}

Assuming a proper class of Woodin limits of Woodin cardinals, the theory of  the Chang model cannot be changed by (set) forcing (\cite[Corollary 3.1.7]{MR2069032}). Hence we obtain the following answer to Q3:  
\begin{theorem}[Woodin]
    Assuming a proper class of Woodin limits of Woodin cardinals, the theory of $\cto$ cannot be changed by forcing. 
\end{theorem}

This emphasizes the relevance of trying to solve Q1. In this respect there is a similarity between $\cto$ and the inner models $C^*$ and $\CAA$.

Let us recall that assuming a proper class of Woodin cardinals, there is no inner model of $\CAA$ with a Woodin cardinal. In sharp contrast to this:

\begin{theorem}\label{manywoodins}If there is a proper class of Woodin cardinals, then $\cto\nsubseteq \CAA$. Moreover, then $\cto$ contains, for every $n$, an inner model with $n$ Woodin cardinals.
\end{theorem}

\begin{proof}For every $n<\omega$, the real
  $M_n^\sharp$ is a $\Pi^1_{n+2}$-singleton (\cite{MR1336414}) and can therefore be defined in, and is an element of, $\cto$. On the other hand, by \cite[Theorem 5.32]{KMV2} and our assumption, the reals of $\CAA$ form a countable $\Sigma^1_3$-set. Hence $M_1^\sharp\notin\CAA$.  If $M_n^\sharp$ is stretched by iterating the top measure inside $\cto$ over the ordinals, an inner model with $n$ Woodin cardinals obtains inside $\cto$.
\end{proof}

\begin{corollary}
    Suppose there are infinitely many Woodin cardinals. Then $$\hc^{\cto}\prec\hc^V.$$
\end{corollary}

\begin{proof} Let us fix $n<\omega$. We already know that  $M_n^\sharp$ is in $\cto$. In fact, $M_n^\sharp(r)$ is in $\cto$ for all reals $r\in\cto$ (\cite{MR4155448}). Moreover, $\hc^{M_n(r)}$ is a $\Sigma_{n+1}$-elementary submodel of $\hc$ for every real $r$ (\cite{MR1336414}). An easy induction now shows that if $a_1,\ldots,a_n$ are reals in $\cto$, then for all $\Sigma_{n+1}$-formulas $\phi(x_1,\ldots,x_m)$ we have 
    $$\cto\models``\hc\models\phi(a_1,\ldots,a_m)"\iff V\models``\hc\models\phi(a_1,\ldots,a_m)".$$
\end{proof}

Note that if $V=\cto$, there cannot be any measurable cardinals. This can be seen with the same proof which shows that  there are no measurable cardinals in $L$ \cite{MR143710}.

Concerning question Q4 we have the following result:

\begin{theorem}\label{mahlo}Assuming a Woodin limit of Woodin cardinals, the cardinal $\omega_1^V$ is strongly Mahlo
in $C^2(\omega)$. 
\end{theorem}

%Open: weakly compact?

%Can do the same for $\omega_2^V$?

\begin{proof}
 Let $j:V\to M$ as in the countable Stationary Tower Forcing $Q_{<\delta}$ built from the Woodin limit $\delta$ of Woodin cardinals. We will show that $(\cto)^M$ is a definable subclass of $V$.
%Recall that when we proved in \cite{MR4290501} that $\omega_1^V$ is strongly Mahlo in $C^*$, we were able to identify in $V$ what the $C^*$ of $M$ is. It is the ``cof $<\delta$" model $C^*_{<\delta}$. Then we used elementarity: $\delta$ is inaccessible in $V$, hence inaccessible in  $C^*_{<\delta}$, whence $\omega_1^V$ is inaccessible in $C^*$.
%Fortunately,  we can likewise identify $\cto$ of $M$ in $V$:
Let us first note that $(C_{\omega_1\omega})^M=(C_{\omega_1\omega})^{V[G]}$ as $M^\omega\subseteq M$ in $V[G]$.
\medskip

Claim 1: $(C_{\omega_1\omega})^{V[G]}\subseteq V[\oR^{V[G]}]$.
\medskip

To prove the claim it suffices to show that every countable sequence of ordinals in $V[G]$ is in $V[\oR^{V[G]}]$.
 By  \cite[Chapter 3]{MR2069032},
 every $\omega$-sequence in $V[G]$ (i.e. in $M$) is generated (over $V$) by forcing $\oP$ that has size $<\delta$.  
This is where we use the assumption that $\delta$ is a limit of Woodin cardinals. As $\omega_1^{V[G]}=\delta$, a generic for $\oP$ can be coded by a real (in $V[G]$). Claim 1 follows.

The set $S$ of regular cardinals below $\delta$ is stationary in $V$, as $\delta$ is a Woodin cardinal. By \cite[Proposition 4.5]{KMV2} the set $S$ is a stationary subset of $\omega_1$ in $V[G]$. Let $\oF$ be the $\aleph_0$-distributive forcing of \cite{MR434818} that adds a generic over $V[G]$ club  $D\subseteq S$.

\medskip

Claim 2: $\oR^{V[G]}=\oR^{V[H]}$, for some $\Col(\omega,<\!\!\delta)$-generic over $V$ filter $H\in V[G][D]$.
\medbreak

The idea of the proof of this claim is essentially included in the proof of $\Sigma^2_1$-absoluteness under CH (see \cite{MR2069032}). We use $D$ to define in $V[G][D]$ a set  $H\subseteq \Col(\omega,<\!\!\delta)$ such that $H$  codes the reals of $V[G]$, i.e.
$\oR^{V[G]}(=\oR^{V[G][D]})=\oR^{V[H]}$. Let $(r_i)_{i<\delta}$, list the reals of $V[G]$.
We define an increasing club sequence  $(\alpha_i)_{i<\delta}$ of elements of $D$ as follows. 
%Since each $\alpha_i$ is countable in $V[G]$, we have $\Col(\omega,<\!\!\alpha_i)$-generics over $V$. 
Suppose we have, for some $i<\delta$, already constructed, for all $j\le i$, $\alpha_j$ and $G_j$ such that $G_j$ is  $\Col(\omega,<\!\!\alpha_j)$-generic over $V$,  and moreover, $r_j\in V[G_{j+1}]$ for $j< i$.
%We define $G_{i+1}$ in such a way that it is $\Col(\omega,<\!\!\alpha_i)$-generic over $V$. 
The real $r_{i}$ is constructed by some forcing $\oU$ over $V$ such that $|\oU|<\delta$. Let $\alpha_{i+1}\in D$ such that $|\oU|<\alpha_{i+1}$ and $\alpha_{i+1}>\alpha_i$. By the universality property of the Levy collapse  $\Col(\omega,<\!\!\alpha_{i+1})$, also known as the Kripke Embedding Theorem,  (see e.g. \cite[Corollary 26.8]{MR1940513}), the poset $\oU$ can be completely embedded into $\Col(\omega,<\!\!\alpha_{i+1})$. Since $\alpha_{i+1}$ is countable in $V[G]$, there is 
in $V[G]$ a $\Col(\omega,<\!\!\alpha_{i+1})$-generic $G_{i+1}$ over $V$, which extends $G_i$ and constructs $r_{i}$. We shall next define $G_i$ and $\alpha_i$ for limit $i$. Let $\alpha_i=\sup\{\alpha_j:j<i\}$. We still have $\alpha_i\in D$, whence $\alpha_i$ is regular in $V$.
% Let $G_i$ be $\Col(\omega,<\!\!\alpha_i)$-generic over $V$ for $i<\xi$, $\xi$ limit. 
Now $G_i=\bigcup_{j<i}G_j$ is $\Col(\omega,<\!\!\alpha_i)$-generic over $V$, because $\Col(\omega,<\!\!\alpha_i)$ has $\alpha_i$-c.c. Finally, let $H=\bigcup_{i<\delta}G_i$.
%So any maximal antichain is actually included in one of the forcings, and then $G_i$ meets it.
Claim 2 is proved.

Now all $\omega$-sequences from $V[G]$ are in $V[H]$. Thus $\cto^{V[H]}=\cto^{V[G]}$. By the homogeneity of Levy 
collapse $\Col(\omega,<\!\!\delta)$, $\hod^{V[H]}=\hod^V$. 
$\cto$ is always a definable subclass of $\hod$. Hence $\cto^{V[G]}=\cto^M$ is definable in $V$. The Woodin cardinal $\delta$ is inaccessible in $V$, hence in the definable subclass $\cto^M$ of $V$, so $\omega_1^V$ is inaccessible in $\cto$ by elementarity. The strong Mahloness of $\omega_1^V$ in $\cto$ follows now as in the proof of \cite[Theorem 5.13]{MR4290501}.   
\end{proof}

We leave it as an open question, whether the above proof can be extended to weak compactness or even measurability of $\omega_1^V$ in $\cto$ or to strong Mahloness in $\cto$ of every regular cardinal of $V$.
  
We use $\ctop$ to denote the modification of $\cto$ in which truth is checked in the next admissible set, that is, $(J'_\alpha,\in,{\tr}\rest\alpha)\models\exists X\phi$ is interpreted as ``there is a countable subset $A$ of $J'_\alpha$ such that the next admissible set of 
$(J'_\alpha,\in,{\tr}\rest\alpha,A)$
satisfies $\phi$.''

\begin{theorem}[\cite{mice}]\label{mice}
   If $V=L^\mu$, then $V=C(aa)$.
\end{theorem}

\begin{theorem}\label{lmu}
    If $V=L^\mu$, then $\ctop=M_\omega$, whence $\ctop$ (and a fortiori also $\cto$) is a proper submodel of $C(aa)$.
\end{theorem}

\begin{proof}
Let us assume $V=L^\mu$.     If we take the ultrapower $M$, the inner model $\ctop$ does not change, as $M^\omega\subseteq M$. 
In $L^\mu$ the Chang model is $M_\omega[\langle \kappa_n:n<\omega\rangle]$, i.e. $\omega$ iterations of $L^\mu$ plus the Prikry-sequence (\cite{MR0409185}). The $\hod$ of $M_\omega[\langle \kappa_n:n<\omega\rangle]$ is simply $M_\omega$, because Prikry-forcing is homogeneous. So $\ctop\subseteq M_\omega$. 

We next prove $M_\omega \subseteq \ctop$. Let us look at the sequence $\langle \kappa_n:n<\omega\rangle$.
The unique measure on $M_\omega$ is generated by the co-finite parts of this sequence. 
 Let us go in the hierarchy generating $\ctop$ past, say,  $\kappa^{+(17)}$. Then we can ask, is there a countable sequence of ordinals so that  we can construct an iterable measure from that. We can do it in the next admissible set i.e. the smallest admissible set containing the sequence and the level of $\ctop$.  We can define the measure of $M_\omega$ by quantifying over countable sets. Inside $\ctop$ we can check whether 
 %from the sequence we get a measure that is iterable.
 the measure is iterable. This defines the measure uniquely, although the countable sequence used is not unique. In this way we have defined the measure of $M_\omega$ in $\ctop$. Then we can define $M_\omega$ as the minimal model which contains that measure.

 Hence, if $V=L^\mu$, then by Theorem~\ref{mice}, $\cto$ is a proper submodel of $\CAA$.
\end{proof}

%OPEN: How far can we push the above?

\begin{theorem}\label{realsofcaaincto}
Assuming a proper class of Woodin cardinals, all reals of $\CAA$ are in $\cto$.
\end{theorem}

\begin{proof}
Under our assumption, there is a $\Sigma^1_3$ well-order $<^*$ of the reals of $\CAA$, and the set of reals of $\CAA$ is a countable set  in $V$ (\cite{KMV2}). This well-order is $\Sigma^1_3$-definable both in $\CAA$ and in $V$. Let its length be $\delta$. The set of reals of $\CAA$ is, as the range of $<^*$, a $\Sigma^1_3$ definable set of reals.  By transfinite induction over $\alpha<\delta$ we can show that the $\alpha^{\mbox{\scriptsize th}}$ real of $\CAA$ is in $\cto$. Note that we can quantify over countable sequences of reals in the construction of $\cto$, whether these   sequences and reals are in $\cto$ or not.
\end{proof}

\begin{theorem}\label{subsetsofomega1}
    $\MM^{++}$ implies that subsets of $\omega_1$ that are in $\CAA$ are in $\cto$.
\end{theorem}

\begin{proof} It is a folklore fact that from $\MM^{++}$ we get a L\"owenheim-Skolem-Tarski Theorem down to $\aleph_1$ for $\CAA$ (this is as in \cite{MR506381}).
    In consequence, every subset of $\omega_1$ in $\CAA$ is on the level $J'_\beta$ of $\CAA$ for some $\beta<\omega_2^V$. Since $\MM$ implies $\delta^1_2=\omega_2$ (\cite[Theorem 9.79]{MR1713438}), there is $Y\subseteq\omega$ such that $((\aleph_1^V)^+)^{L[Y]}>\beta$. Now
    $$(J'_\beta)^{L[Y^\sharp]}=J'_\beta.$$
    We do not know whether $Y$ is in $\cto$ but we do not need this  because we can use $Y$ and $Y^\sharp$ as predicates, as values of bound second order variables. We can recover $J'_\beta$, and thereby $X$, inside $\cto$ by means of $Y^\sharp$.
    
    %$\hod^{H_{\omega_2}}$. Now we note that pmax forcing is homogeneous and $\hod^{H_{\omega_2}}\subseteq \hod^{L(\oR)}\subseteq \cto$. So every subset of $\omega_1$ in $\CAA$ is in $\cto$.
\end{proof}

\begin{conjecture} Assuming appropriate large cardinals, we have $\CAA\subset \cto$.
\end{conjecture}

Let $\mathcal{F}(\lambda)$ denote the club filter of $\mathcal{P}_{\omega_1} (\lambda^\omega)$.

\begin{definition}[Woodin \cite{woodin}]
    We say that CM$^+$ \emph{exists} if for all $\lambda$ the following holds. For all $Z \subseteq \mathcal{P}_{\omega_1} (\lambda^\omega)$, if $Z \in L(\lambda^\omega)[\mathcal{F}(\lambda)]$ then either
\begin{enumerate}
    \item $Z \in\mathcal{F}(\lambda)$ or
\item $\mathcal{P}_{\omega_1} (\lambda^\omega)\setminus Z \in \mathcal{F}(\lambda)$.
\end{enumerate}
\end{definition}

\begin{theorem}[Woodin \cite{woodin}]\label{woodin}
   Assume there is a proper class of Woodin limits of Woodin cardinals. Then  CM$^+$ exists. 
\end{theorem}

\begin{definition}\label{changplus}
We say that the inner model $C^2(\omega,\aaq)$ satisfies \emph{Club Determinacy}, if every level $(J'_\alpha,\in,{\tr}\rest\alpha)$ of the construction of $C^2(\omega,\aaq)$ satisfies: 
$$\label{cdehto}
(J'_\alpha,\in,{\tr}\rest\alpha)\models
\forall\bx[\aaqs\phi(\bx,\bt,s)\vee
\aaqs\neg\phi(\bx,\bt,s)],
$$ where $\phi(\bx,\bt,s)$ is any formula in $\LL^2(\omega,\aaq)$ and $\bt$ is a finite sequence of countable subsets of $J'_\alpha$. Similarly for $\chaa$. 
\end{definition}

\begin{theorem}Assuming a proper class of Woodin limits of Woodin cardinals, 
    $\ctoaa$ and $\chaa$ satisfy Club Determinacy. 
    %\footnote{Can we make sure $\delta^1_2=\omega_2$?}
\end{theorem}

\begin{proof}We use Theorem~\ref{woodin}. 
    We aim at reducing the statement of Club Determinacy of  $\ctoaa$ to the statement ``CM$^+$ exists" of Definition~\ref{changplus}. To this end, let us assume 
$$
S=\{s\in \Pw_{\omega_1}(J'_\alpha) : (J'_\alpha,\in,{\tr}\rest\alpha)\models\phi(\ba,s,\bt)\},
$$
 where $\ba$ is a finite sequence of elements of $J'_\alpha$ and $\bt$ is a finite sequence of countable subsets of $J'_\alpha$, violates Club Determinacy. We note that $S$ is in $L(\lambda^\omega)[\mathcal{F}(\lambda)]$, where $\lambda=|J'_\alpha|$. This contradicts the assumption ``CM$^+$ exists''.
\end{proof}

\begin{corollary}\label{measurable}
    Assume a proper class of Woodin limits of Woodin cardinals. Then every regular  cardinal of $V$ is measurable in $\ctoaa$.  
\end{corollary}

%[Do we get forcing absoluteness from Woodin's result?]

\begin{proof}This is as the proof of Theorem 5.1 in \cite{KMV2}. Fix a regular uncountable $\kappa$.
    For  $\alpha$ big enough for $J'_\alpha$ to contain all subsets of $\kappa$ in $\ctoaa$, consider the 
    $\kappa$-complete normal filter:
$$\F=\{X\subseteq\kappa : X\in J'_\alpha, (J'_\alpha,\in,\tr\rest\alpha)\models \aaqs(\sup(s\cap\kappa)\in X)\}.$$
%This set is clearly in $J'_{\alpha+1}$.
%
Suppose $X\subseteq\kappa$ is in $\ctoaa$.  Since $(J'_\alpha,\in,\tr\rest\alpha)$ satisfies Club Determinacy, it satisfies
$\aaqs(\sup(s\cap\kappa)\in X)$ or $\aaqs(\sup(s\cap\kappa)\notin X).$
If the first, then  $X\in\F$. If the second, then $\kappa\setminus X\in \F$. 
\end{proof}

\begin{lemma}\label{equivCol}
    If there is a proper class of Woodin limits of Woodin cardinals, then for every Woodin limit $\delta$ of Woodin cardinals, $\chaa^V\equiv
    \chaa^{V^{\Col(\omega,<\delta)}}.$ 
\end{lemma}

\begin{proof}
    Let  $G$ be a generic for the countable stationary tower forcing at $\delta$. Let $j:V\to M\subseteq V[G]$ be the associated elementary embedding. Here ${}^\omega M\subseteq M$. By Theorem~\ref{woodin} the model $\chaa^V$ satisfies Club Determinacy. Hence also $M$ satisfies Club Determinacy. 
    
    \smallskip

    \noindent{\bf Claim 1:} $\chaa^M=\chaa^{V[G]}$.
    \smallskip

    Proof of the claim: We use induction on the construction of $\chaa$, appealing to  Club Determinacy and ${}^\omega M\subseteq M$ on the way. As in \cite[Theorem 4.12]{KMV2}, proving $\chaa^M=\chaa^{V[G]}$ boils down to proving the following: If $N\in \chaa^M$, $(a\in{}^\omega M)^M$, and $\phi(x)$ is in $\mathcal{L}_{\omega_1\omega}(\aaq)\cap M$, then
    \begin{equation}
        M\models``N\models\phi(a)"\iff V[G]\models``N\models\phi(a)"
    \end{equation}
    We use induction on $\phi(x)$. The only case needing attention is the case that $\phi(x)$ is $\aaqs\psi(s,x)$, and the claim holds for $\psi(s,x)$ already. Suppose first $M\models``N\models\aaqs\psi(s,a)"$. Let $C\in M$ be a club of countable subsets $P$ of $N$ in $M$ such that $M\models``(N,P)\models\psi(P,a)"$. By induction hypothesis, for all $P\in C$, $V[G]\models``(N,P)\models\psi(P,a)"$ and $C$ is club in $V[G]$, too. Hence $V[G]\models``N\models\aaqs\psi(s,a)"$. On the other hand, assume $M\not\models``N\models\aaqs\psi(s,a)"$. By Club Determinacy, $M\models``N\models\aaqs\neg\psi(s,a)"$. Let $C\in M$ be a club of countable subsets $P$ of $N$ in $M$ such that $M\models``(N,P)\models\neg\psi(P,a)"$. By induction hypothesis, for all $P\in C$, $V[G]\models``(N,P)\models\neg\psi(P,a)"$. Hence $V[G]\models``N\models\aaqs\neg\psi(s,a)"$, and $V[G]\not\models``N\models\phi(a)"$ follows.
    Claim 1 is proved.

    Let $S$ be the set of regular cardinal $<\delta$ in $V$. The set $S$ is stationary, as $\delta$ is Mahlo. By \cite{KMV2}, the set $S$ is still stationary in $V[G]$. Of course, in $V[G]$ it is a set of countable ordinals. Let $H$ be generic over $V[G]$ for the standard po-set for forcing a club subset for $S$. Note that $M, V[G]$ and $V[G][H]$ have the same countable sequences of ordinals. We can use $H$, as in the proof of Theorem~\ref{mahlo}, to construct a $\Col(\omega,<\!\!\delta)$-generic $T$ over $V$ such that $V[T]$ and $V[G][H]$ have the same countable sets of ordinals. 

    \smallskip

    \noindent{\bf Claim 2:} $\chaa^M=\chaa^{V[T]}$.
    \smallskip

    Proof of the claim: We use induction on the construction of $\chaa$, appealing to  Club Determinacy. The proof is like the proof of Claim 1.
    Claim 2 is proved. 
    
    We now return to the proof of Lemma~\ref{equivCol}:  The stationary tower embedding $j:V\to M$ induces $j:\chaa^V\to \chaa^M=\chaa^{V[T]}$. Finally, we note that the po-set $\Col(\omega,<\!\!\delta)$ is homogeneous and therefore it does not matter which generic we use.
\end{proof}

\begin{theorem}[Woodin \cite{woodin}]
    Suppose there is a proper class of Woodin limits of Woodin cardinals. Suppose $\oP$ is a forcing notion and $G$ is $\oP$-generic over $V$. Then $\chaa^V\equiv\chaa^{V[G]}$.
\end{theorem}

\begin{proof}
    Suppose $\oP$ is a forcing notion and $H$ if $\oP$-generic over $V$. Let $\delta$ be a Woodin limit of Woodin cardinals such that $\delta>|\oP|$. Still, $\delta$ is Woodin in $V[H]$ and  $V[H]$ still has a proper class of Woodin limits of Woodin cardinals. Let $G$ be $\Col(\omega,<\delta)$-generic over $V[H]$.  By Lemma~\ref{equivCol}, $\chaa^V\equiv
    \chaa^{V[G]}$. Similarly,  
    $\chaa^{V[H]}\equiv\chaa^{V[H][G]}$. By well-known properties of the Levy-collapse, there is a $\Col(\omega,<\delta)$-generic $G'$ over $V$ such that $V[H][G]=V[G']$. Hence by the homogeneity of the collapse, $\chaa^V\equiv\chaa^{V[H]}$.
\end{proof}

    \begin{corollary}
       Suppose there is a proper class of Woodin limits of Woodin cardinals. Then the theory of $\ctoaa$ cannot be changed by forcing. 
    \end{corollary}
    
This raises, among other things, the question whether a proper class of Woodin limits of Woodin cardinals decides CH in $\ctoaa$.

\section{The inner model $\hod_1$.\label{hod-one}}

Recall that $\hod_1$ is the inner model $C(\Sigma^1_1)$, where $\Sigma^1_1$ is the  fragment of second order logic consisting of formulas of the form
$\exists R_1\ldots\exists R_n\phi,$
where $\phi$ is first order and $R_1,\ldots,R_n$ are second order relation variables. It was proved in \cite{MR4290501} that the equation $\hod=\hod_1$ is independent of ZFC. We refine this result here by showing that adding large cardinal axioms to ZFC does not help to decide $\hod=\hod_1$.

\begin{theorem}\label{hodhod1}
The statement $\hod=\hod_1$ is independent from ZFC plus ``There is a supercompact  cardinal".
\end{theorem}

\begin{proof}
In \cite{MR4290501} this was proved by forcing over $L$. Naturally the question arises whether  some large cardinal can imply a definite answer to this question. We refine the result from \cite{MR4290501} by showing that for many large cardinals axioms do not settle this problem. For concreteness we choose to show it for 
supercompact cardinals, but the  method of proof applies to many large cardinals notions. Essentially we need that, consistently with the given large cardinal, the universe can be coded by the function $\lambda\rightarrow 2^\lambda$.  The proof follows closely the argument of \cite[Section 7]{MR4290501}.

We show that both $\hod=\hod_1$ and $\hod\neq \hod_1$ are consistent with a supercompact cardinal, assuming the consistency of the latter.

$\hod=\hod_1$: The proof of the consistency of $V=\hod$ with a supercompact cardinal (\cite{MR0540771}) gives the same result for $V=\hod_1$ as in the model of \cite{MR0540771} every set of ordinals is coded by the class of ordinals $\alpha$ such that $2^{\aleph_\alpha}=\aleph_{\alpha+3}$ which is captured by $\hod_1$.

$\hod\neq \hod_1$: The proof follows directly the proof of \cite[Theorem 7.6]{MR4290501}. Using  \cite[Theorem 21]{MR0540771},  we can assume that $V=\hod_1$ and that we have cardinal $\delta$ which is supercompact. Also note that if in the coding process of \cite{MR0540771} we make sure to code every set of ordinals unboundedly many times, we get that  the model $V$ is a definable subclass of $(\hod_1)^{V^\oP}$ for every forcing extension of $V$ by a (set) forcing notion $\oP$. Here, by ``definable'' we mean ``definable with no parameters''.

Let $\kappa_n$, $n<\omega$, be the smallest weakly compact cardinals in $V$ and let $\kappa=\sup_{n<\omega}\kappa_n$. For $n<\omega$ let $\oP_n$ be the reverse  Easton support iteration of adding one Cohen subset to any inaccessible in the interval $(\kappa_{n-1}, \kappa_n]$, where we interpret $\kappa_{-1}$ as $0$. Let $\oP^\ast$ be the full support product of the $\oP_n$'s for $n<\omega$. Let $V_1$ be the forcing extension of $V$ by $\oP^\ast$. Note that each $\kappa_n$, $n<\omega$, are still weakly compact in $V_1$. Note by our assumption about $V$ that $(\hod_1)^{V_1}=\hod ^{V_1}=V$. This follows since $V=(\hod)^{V_1}$, $V_1$ is an extension of $V$ by a homogeneous forcing, and $V\subset (\hod_1)^{V_1}$. Also note that $\oP_0$ introduced into $V_1$ a Cohen real $a$.

For each $n<\omega$, let $\oT_n$ in $V^{\oP_n}$ be the forcing that kills the weak compactness of $\kappa_n$ by adding a Souslin tree on it (see \cite{MR495118}). We force over $V_1$  with $\prod_{n\in a} \oT_n$, killing the weak compactness of $\kappa_n$ for $n\in a$. Denote the resulting model by $V_2$. Let $\oT_n^\ast$ be the resulting Soulin tree.  Note that $\kappa_n$ for $n\not\in a$ are still weakly compact, since $\kappa_n$ is weakly compact in $V^{\oP_n}$ and $V_2$ is a an extension of $V^{\oP_n}$ by a forcing which is a product of a forcing of size less than $\kappa_n$ and a forcing which is in $V$ and is  $\kappa_n^+$-strategically closed in $V$. Also, as in \cite{MR495118}, for $n\in a$, $P_n\ast \oT^\ast_n$ is forcing equivalent to $\oP_n$. Hence in $V_2$ the cardinal $\kappa_n$ is weakly compact iff $n\not\in a$. Since, by assumption $V\subseteq (\hod)^{V_2}$, we have  
$\{\kappa_n|n<\omega\}\in (\hod)^{V_2}$.  Hence $a\in (\hod)^{V_2}$.

We claim that $V_2$ is our model of a supercompact cardinal with $\hod\neq \hod_1$. We can immediately observe that $\delta$ is still supercompact in $V_2$, as we forced over $V$ by a forcing of size less that $\delta$. We shall obtain our result if we show that $a\not\in (\hod_1)^{V_2}$, so in $V_2$ we have $a\in \hod-\hod_1$. This follows from: 

\begin{claim}
$(\hod_1)^{V_2}=(\hod_1)^{V_1}=V$.
\end{claim}

In order to prove the claim,we force over $V_2$  with the full support product of $\oT^\ast_n$ for $n\in a$. Let $V_3$ be the resulting model. Note that $V_3$ is also forcing extension of $V$ by $\oP^\ast$.  As in \cite[Theorem 7.6]{MR4290501}, we show by induction on the levels that all the levels of the construction of $\hod_1$ are the same in $V_d$ for $d= 1,2,3$ and all these levels belong to $V$.  This will follow from:

\begin{subclaim} Let $M\in V$, Let $\Phi(\vec{x})$ be a $\Sigma^1_1$ formula. Let $\vec{b}\in M$ then the following are equivalent
\begin{enumerate}
  \item $ (M\models\Phi(\vec{b}))^{V_1}$ 
  \item $ (M\models\Phi(\vec{b}))^{V_2}$ 
  \item $ (M\models\Phi(\vec{b}))^{V_3}$ 
\end{enumerate}

\end{subclaim}
 \begin{proof} [subclaim]

By the persistency of existential second order formulas in extensions of models of set theory, $(1)\rightarrow (2) \rightarrow (3)$. So we are left with showing $(3)\rightarrow (1)$.
Suppose that (3) holds. $V_3$ is an extension of $V$ by $\oP^\ast$. If $V_3\models ``M\models \Phi(\vec{b})$'', then, since $\oP^\ast$ is a homogeneous forcing notion and $M,\vec{b}\in V$ , every condition in $\oP^\ast$ forces that $M\models \Phi(\vec{b})$. $V_1$ is also an extension of $V$ by $\oP^\ast$, so $V\models ``M\models \Phi(\vec{b})$'', which is item (1) above. The Subclaim, and thereby the Claim are proved.
\end{proof}

 Theorem~\ref{hodhod1} is proved. \end{proof}

\section{Open questions}

\begin{enumerate}
    \item Is $\omega_1^V$  measurable
in $C^2(\omega)$, under large cardinal assumptions? The answer is positive in a strong sense for both $\CAA$ and $\ctoaa$.
 \item Is every uncountable regular cardinal of $V$ strongly Mahlo
in $C^2(\omega)$, under large cardinal assumptions? The answer is positive if $C^2(\omega)$ is replaced by  $C^*$,  $\CAA$, or $\ctoaa$.
    \item Does $\CAA\subseteq\cto$ follow from large cardinal assumptions? The answer is positive for reals. It is positive also for subsets of $\omega_1^V$, assuming $\MM^{++}$.
    \item Does $\cto$ satisfy $\GCH$, under large cardinal assumptions? The answer is positive for $\CAA$ up to $\omega_1^V$ (\cite{KMV2}), and also in general by a result of Goldberg and Steel \cite{GS}.
    \item What large cardinals can exist in $\cto$? Under appropriate large cardinal assumptions, there are higher measurable cardinals in $\CAA$ by an unpublished result of Goldberg and Rajala.
    \item Given $n$, is $(\neg\CH)^{\cto}$  consistent with $n$ Woodin cardinals, assuming the existence of $M_n$? 
    \item Does $\cto\ne\ctop$ follow from large cardinals?
    
\end{enumerate}

\bibliographystyle{plain}
\bibliography{Magidor_Vaananen_C2_omega}

\begin{thebibliography}{10}

\bibitem{BF}
J.~Barwise and S.~Feferman, editors.
\newblock {\em Model-theoretic logics}.
\newblock Perspectives in Mathematical Logic. Springer-Verlag, New York, 1985.

\bibitem{MR434818}
J.~E. Baumgartner, L.~A. Harrington, and E.~M. Kleinberg.
\newblock Adding a closed unbounded set.
\newblock {\em J. Symbolic Logic}, 41(2):481--482, 1976.

\bibitem{MR506381}
Shai Ben-David.
\newblock On {S}helah's compactness of cardinals.
\newblock {\em Israel J. Math.}, 31(1):34--56, 1978.

\bibitem{MR280357}
C.~C. Chang.
\newblock Sets constructible using {$L\sb{\kappa \kappa }$}.
\newblock In {\em Axiomatic {S}et {T}heory ({P}roc. {S}ympos. {P}ure {M}ath., {V}ol. {XIII}, {P}art {I}, {U}niv. {C}alifornia, {L}os {A}ngeles, {C}alif., 1967)}, volume XIII, Part I of {\em Proc. Sympos. Pure Math.}, pages 1--8. Amer. Math. Soc., Providence, RI, 1971.

\bibitem{MR0409185}
Patrick Dehornoy.
\newblock Solution d'une conjecture de {B}ukovsky.
\newblock {\em C. R. Acad. Sci. Paris S\'er. A-B}, 281(20):Ai, A821--A824, 1975.

\bibitem{MR539973}
M.~A. Dickmann.
\newblock {\em Large infinitary languages}, volume Vol. 83 of {\em Studies in Logic and the Foundations of Mathematics}.
\newblock North-Holland Publishing Co., Amsterdam-Oxford; American Elsevier Publishing Co., Inc., New York, 1975.

\bibitem{mice}
Gabriel Goldberg, Menachem Magidor, Ralf Schindler, and John Steel.
\newblock The mice of ${C}(aa^+)$.
\newblock Under preparation.

\bibitem{GS}
Gabriel Goldberg and John Steel.
\newblock The structure of {C}(aa).
\newblock To appear in Monatshefte für Mathematik.

\bibitem{MR0465866}
Leo Harrington.
\newblock Long projective wellorderings.
\newblock {\em Ann. Math. Logic}, 12(1):1--24, 1977.

\bibitem{MR1940513}
Thomas Jech.
\newblock {\em Set theory}.
\newblock Springer Monographs in Mathematics. Springer-Verlag, Berlin, millennium edition, 2003.

\bibitem{MR4290501}
Juliette Kennedy, Menachem Magidor, and Jouko V\"a\"an\"anen.
\newblock Inner models from extended logics: {P}art 1.
\newblock {\em J. Math. Log.}, 21(2):Paper No. 2150012, 53, 2021.

\bibitem{KMV2}
Juliette Kennedy, Menachem Magidor, and Jouko V\"a\"an\"anen.
\newblock Inner models from extended logics: {P}art 2, 2025.
\newblock J. Math. Log., Paper No. 2550009, 47.

\bibitem{MR495118}
Kenneth Kunen.
\newblock Saturated ideals.
\newblock {\em J. Symbolic Logic}, 43(1):65--76, 1978.

\bibitem{MR2069032}
Paul~B. Larson.
\newblock {\em The stationary tower}, volume~32 of {\em University Lecture Series}.
\newblock American Mathematical Society, Providence, RI, 2004.
\newblock Notes on a course by W. Hugh Woodin.

\bibitem{MR0540771}
Telis~K. Menas.
\newblock Consistency results concerning supercompactness.
\newblock {\em Trans. Amer. Math. Soc.}, 223:61--91, 1976.

\bibitem{MR4155448}
Sandra M\"uller, Ralf Schindler, and W.~Hugh Woodin.
\newblock Mice with finitely many {W}oodin cardinals from optimal determinacy hypotheses.
\newblock {\em J. Math. Log.}, 20:1950013, 118, 2020.

\bibitem{MR0281603}
John Myhill and Dana Scott.
\newblock Ordinal definability.
\newblock In {\em Axiomatic {S}et {T}heory ({P}roc. {S}ympos. {P}ure {M}ath., {V}ol. {XIII}, {P}art {I}, {U}niv. {C}alifornia, {L}os {A}ngeles, {C}alif., 1967)}, volume XIII, Part I of {\em Proc. Sympos. Pure Math.}, pages 271--278. Amer. Math. Soc., Providence, RI, 1971.

\bibitem{MR143710}
Dana Scott.
\newblock Measurable cardinals and constructible sets.
\newblock {\em Bull. Acad. Polon. Sci. S\'er. Sci. Math. Astronom. Phys.}, 9:521--524, 1961.

\bibitem{MR1336414}
J.~R. Steel.
\newblock Projectively well-ordered inner models.
\newblock {\em Ann. Pure Appl. Logic}, 74(1):77--104, 1995.

\bibitem{woodin}
W.~Hugh Woodin.
\newblock Determinacy and generic absoluteness.
\newblock To appear.

\bibitem{MR1713438}
W.~Hugh Woodin.
\newblock {\em The axiom of determinacy, forcing axioms, and the nonstationary ideal}, volume~1 of {\em De Gruyter Series in Logic and its Applications}.
\newblock Walter de Gruyter \& Co., Berlin, 1999.

\end{thebibliography}

\end{document}